\documentclass[11pt]{amsart}

\usepackage[top=1.5in, bottom=1in, left=0.9in, right=0.9in]{geometry}

\usepackage{amscd,amsmath,amssymb,fancyhdr,color}
\usepackage[utf8]{inputenc}
\usepackage{amsfonts}
\usepackage{amsthm}
\usepackage{accents}
\usepackage{graphicx}
\usepackage{float}
\usepackage{subcaption}
\usepackage{verbatim}
\usepackage{indentfirst}
\usepackage{dsfont}
\usepackage{tikz}
\usepackage{tikz-cd}
\usetikzlibrary{matrix}
\usepackage[all]{xy}
\usepackage{enumerate}
\usepackage{extarrows}
\usepackage{csquotes}

\usepackage{scalerel,stackengine}
\stackMath
\newcommand\reallywidehat[1]{%
	\savestack{\tmpbox}{\stretchto{%
			\scaleto{%
				\scalerel*[\widthof{\ensuremath{#1}}]{\kern-.6pt\bigwedge\kern-.6pt}%
				{\rule[-\textheight/2]{1ex}{\textheight}}%WIDTH-LIMITED BIG WEDGE
			}{\textheight}% 
		}{0.5ex}}%
	\stackon[1pt]{#1}{\tmpbox}%
}

\usepackage{faktor}

\usepackage{BOONDOX-uprscr}

\usepackage[backref=page]{hyperref}
\renewcommand*{\backref}[1]{}
\renewcommand*{\backrefalt}[4]{%
	\ifcase #1 (Not cited.)%
	\or        (Cited on page~#2.)%
	\else      (Cited on pages~#2.)%
	\fi}

\hypersetup{
	colorlinks   = true,
	citecolor    = magenta
}

\newcommand{\K}{K\"ahler}

\DeclareMathOperator{\Ker}{Ker}

\numberwithin{equation}{section}

\def\eqref#1{(\ref{#1})}

\newcommand{\N}{{\mathbb N}}
\newcommand{\Z}{{\mathbb Z}}
\newcommand{\C}{{\mathbb C}}
\newcommand{\R}{{\mathbb R}}
\newcommand{\Q}{{\mathbb Q}}

\renewcommand{\H}{{\mathbb H}}

\newcommand{\delb}{\overline{\partial}}

\def\1{\sqrt{-1}\:}

\newcommand{\cntrct}                % contraction with a vector field
{\hspace{2pt}\raisebox{1pt}{\text{$\lrcorner$}}\hspace{2pt}}

\newcommand{\End}{\operatorname{End}}

\newcommand{\Id}{\operatorname{Id}}

\newcommand{\Aut}{\operatorname{Aut}}

\renewcommand{\dim}{\operatorname{dim}}

\newcommand{\coker}{\operatorname{coker}}

\newcommand{\Spec}{\operatorname{Spec}}

\renewcommand{\Im}{\operatorname{Im}}

\newcommand{\ie}{{\em i.e. }}

\renewcommand{\to}{\longrightarrow}

\newcounter{Mycounter}[section]
\newcounter{lemma}[section]
\newcounter{claim}[section]
\newcounter{sublemma}[section]
\newcounter{corollary}[section]
\newcounter{theorem}[section]
\newcounter{conjecture}[section]
\newcounter{proposition}[section]
\newcounter{definition}[section]
\newcounter{example}[section]
\newcounter{remark}[section]
\newcounter{problem}[section]
\newcounter{question}[section]
\makeatletter

\@addtoreset{equation}{section}

\@addtoreset{footnote}{section}

\makeatother

\makeatletter
\DeclareRobustCommand*{\mfaktor}[3][]
{
	{ \mathpalette{\mfaktor@impl@}{{#1}{#2}{#3}} }
}
\newcommand*{\mfaktor@impl@}[2]{\mfaktor@impl#1#2}
\newcommand*{\mfaktor@impl}[4]{
	\settoheight{\faktor@zaehlerhoehe}{\ensuremath{#1#2{#3}}}%
	\settoheight{\faktor@nennerhoehe}{\ensuremath{#1#2{#4}}}%
	\raisebox{-0.5\faktor@zaehlerhoehe}{\ensuremath{#1#2{#3}}}%
	\mkern-4mu\diagdown\mkern-5mu%
	\raisebox{0.5\faktor@nennerhoehe}{\ensuremath{#1#2{#4}}}%
}
\makeatother

\usetikzlibrary{arrows,chains,matrix,positioning,scopes}

\makeatletter
\tikzset{join/.code=\tikzset{after node path={%
			\ifx\tikzchainprevious\pgfutil@empty\else(\tikzchainprevious)%
			edge[every join]#1(\tikzchaincurrent)\fi}}}
\makeatother

\tikzset{>=stealth',every on chain/.append style={join},
	every join/.style={->}}

\makeatletter
\newtheorem*{rep@theorem}{\rep@title}
\newcommand{\newreptheorem}[2]{%
	\newenvironment{rep#1}[1]{%
		\def\rep@title{\ref{##1}}%
		\begin{rep@theorem}}%
		{\end{rep@theorem}}}
\makeatother

\newreptheorem{theorem}{Theorem}

\begin{document}
	
	\newpage
	
	\title[Dolbeault cohomology of Endo-Pajitnov manifolds]{Dolbeault cohomology of Endo-Pajitnov manifolds}

  	\author{Liviu Ornea}
	\address{Liviu Ornea \newline
		\textsc{\indent University of Bucharest, Faculty of Mathematics and Computer Science\newline 
			\indent 14 Academiei Str., Bucharest, Romania \newline
			\indent \indent and \newline
			\indent Institute of Mathematics ``Simion Stoilow'' of the Romanian Academy\newline 
			\indent 21 Calea Grivitei Street, 010702, Bucharest, Romania}}
	\email{lornea@fmi.unibuc.ro; liviu.ornea@imar.ro}

	\author{Miron Stanciu}
	\address{Miron Stanciu \newline
		\textsc{\indent University of Bucharest, Faculty of Mathematics and Computer Science\newline 
			\indent 14 Academiei Str., Bucharest, Romania \newline
			\indent \indent and \newline
			\indent Institute of Mathematics ``Simion Stoilow'' of the Romanian Academy\newline 
			\indent 21 Calea Grivitei Street, 010702, Bucharest, Romania}}
	\email{miron.stanciu@fmi.unibuc.ro; miron.stanciu@imar.ro}
	
	\thanks{Both authors are partially supported by the PNRR-III-C9-2023-I8 grant CF 149/31.07.2023 Conformal Aspects of Geometry and Dynamics. \\[.1cm]
	}
	
	\date{\today}

	\begin{abstract}
		Endo-Pajitnov manifolds are compact non-K\"ahler manifolds which generalize the Inoue surfaces $S_M$ to higher dimensions. In this paper, we compute their Dolbeault cohomology, show that they satisfy the Hodge decomposition at the level of dimensions and prove that they are rigid with respect to small deformations.
	\end{abstract}
	
	\maketitle
	
	\hypersetup{linkcolor=blue}
	\tableofcontents

	\section{Introduction}
	\label{sec:introduction}
	In the realm of non-K\"ahler geometry, Inoue surfaces (\cite{inoue}) play a prominent role. They are affine surfaces with no curves and no non-constant meromorphic functions, bear a solvmanifold structure and their construction is related to number theory. From the metric viewpoint, almost all of them admit locally conformally K\"ahler metrics (LCK), see \cite{OV:book}. As such, various attempts were made to generalize their construction to arbitrary dimension. Much studied nowadays are the Oeljeklaus-Toma manifolds (OT-manifolds), \cite{ot}, which generalize the Inoue surfaces $S_M$ associated to a matrix $M\in \mathrm{SL}(3,\Z)$. OT-manifolds are non-K\"ahler and share many of the analytic properties of the Inoue surfaces, but they are not always LCK (\cite{vuliOT}). More recently, another generalization of the same Inoue surfaces $S_M$ was proposed by Endo and Pajitnov (\cite{pajitnov1}, see Subsection \ref{EP_Subsection} for details). These manifolds  have a rich analytic and metric geometry (\cite{ciulicaotimanstanciu,ciulica1,ciulica2}). In particular, they are non-K\"ahler solvmanifolds, never LCK (in dimension greater than $2$), never balanced, but sometimes admit other special non-\K \ metrics (\cite{ciulicaotimanstanciu}).

    In this note, we continue the study of Endo-Pajitnov manifolds by computing their Dolbeault cohomology (Section \ref{sec:dolbeault}). Our computation makes explicit use of the fact that Endo-Pajitnov manifolds are odd-dimensional torus bundles over the circle  (\cite[Proposition 2.10]{pajitnov1}) and uses a strong dispersivity property for Cousin groups, much like in \cite{alextoma}. As a corollary, we show that Endo-Pajitnov manifolds satisfy the Hodge decomposition at a dimensional level. In Section \ref{sec:liealg}, we compare the computed cohomology with the invariant one coming from the solvmanifold structure of the manifold. Lastly, in Section \ref{sec:deformations}, with the same techniques used for the Dolbeault cohomology, we prove that Endo-Pajitnov manifolds are rigid with respect to small deformations. 
	
	\section{Preliminaries}
	\label{sec:prelim}

    \subsection{Endo-Pajitnov manifolds}\label{EP_Subsection}
	We begin by recalling the construction of the Endo-Pajitnov manifolds, following \cite{pajitnov1} and the notations from \cite{ciulicaotimanstanciu}.
	
	Let $n \geq 1$ and let $M \in \textrm{SL}(2n+1, \Z)$ be a matrix whose distinct eigenvalues are $\alpha, \beta_1, \ldots, \beta_k, \overline{\beta_1}, \ldots, \overline{\beta_k}$, where $\alpha>0$, $\alpha\neq 1$ has multiplicity $1$ and $\Im(\beta_j)>0$. 
	
	Denote by $V$ the eigenspace corresponding to $\alpha$ and denote the generalized eigenspaces by
	\[
	W(\beta_j)=\{x \in \C^{2n+1}\ | \ \exists \ N \in \N \text{ such that } (M-\beta_j I)^Nx=0\},
	\]
	Then, with the notations $W=\bigoplus\limits_{j=1}^k W(\beta_j)$ and $\overline W=\bigoplus\limits_{j=1}^k W(\overline{\beta_j})$, we have the decomposition $$\C^{2n+1}=V\oplus W\oplus \overline W.$$
	
	Let $a\in \R^{2n+1}$ be a non-zero eigenvector corresponding to $\alpha$ and let $\{b_1,\ldots,b_n\}$ be a basis of $W$. We write $a=(a^1,a^2,\ldots,a^{2n+1})^T, b_i=(b_i^1,b_i^2,\ldots,b_i^{2n+1})^T, 1\leq i\leq n$. For any $1\leq i\leq 2n+1$, put
	\[
	u_i=(a^i,b_1^i,\ldots,b_n^i)\in \R\times \C^n \simeq \R^{2n+1}.
	\]
	Since $\{a,b_1,\ldots,b_n,\overline{b_1},\ldots,\overline{b_n}\}$ is a basis of $\C^{2n+1}$, the vectors $u_1,\ldots,u_{2n+1}$ are linearly independent over $\R$.
	
	Lastly, let $f_M:W\to W$ be the restriction of the multiplication by $M$ to $W$, and let $R=(r_{ij})$ be the matrix of $f_M$ with respect to the basis $\{b_1,\ldots,b_n\}$. 
	
	Using the above data extracted from the matrix $M$, we define a manifold by considering the action of the automorphisms $g_0,g_1,\ldots,g_{2n+1}:\H\times \C^n\to \H\times \C^n$, given by
	\[
	g_0(w,z)=(\alpha w,R^Tz), \qquad g_i(w,z)=(w,z)+u_i, \quad 1\leq i\leq 2n+1.
	\]
	Note that these are well defined because $\alpha>0$ and the first component of each $u_i$ is real. Let $G_M$ be the subgroup of $\Aut(\H\times \C^n)$ generated by $g_0,g_1,\ldots,g_{2n+1}$. Endo and Pajitnov (\cite{pajitnov1}) proved that $G_M$ acts freely and properly discontinuously on $\H\times \C^n$ and that the quotient
	\[
	T_M:=G_M\backslash (\H\times \C^n)
	\]
	is a compact complex manifold of complex dimension $n+1$, with $\pi_1(T_M)\simeq G_M$. We call manifolds obtained from this construction Endo-Pajitnov manifolds.
	
	\begin{remark}
		\label{rmk:choice_basis}
		It is easy to see that the biholomorphism class of $T_M$ does not depend on the choice of basis $\{b_1,\ldots,b_n\}$ of $W$. Indeed, if $b'_1,\ldots,b'_n$ is obtained from $b_1,\ldots,b_n$ by a matrix $C\in \textrm{GL}(n,\C)$, then the change of coordinates $(w,z)\mapsto (w,C^Tz)$ conjugates the corresponding actions. 
		
		On the other hand, the choice of $a$ \textit{does} change the construction, but in a sense only up to a sign. Indeed, since the $\alpha$-eigenspace is a real line, replacing $a$ by a positive multiple $ca$, $c>0$, gives a biholomorphic manifold by the change of coordinates $(w,z)\mapsto (cw,z)$. 
	\end{remark}
	
	\begin{remark}
		\label{rmk:basic_properties}
		We list below a few additional properties of Endo-Pajitnov manifolds that are known in the literature:
		\begin{itemize}
			\item All $T_M$ manifolds are solvmanifolds (\cite[Theorem 3.1]{ciulicaotimanstanciu});
			\item If $M$ is diagonalizable, then some $T_M$ are biholomorphic to OT manifolds (\cite[Proposition 5.3]{pajitnov1});
			\item If $M$ is not diagonalizable, then $T_M$ cannot be biholomorphic to any OT manifold (\cite[Proposition 5.6]{pajitnov1});
			\item $T_M$ does not admit locally conformally \K \ metrics unless $n = 1$ (\cite[Proposition 4.6]{pajitnov1}, \cite[Proposition 5.5]{ciulicaotimanstanciu}) or balanced metrics (\cite[Proposition 5.3]{ciulicaotimanstanciu}), but always admits locally conformally balanced metrics (\cite[Proposition 5.4]{ciulicaotimanstanciu}). Additionally, if the matrix M has some quantifiable algebraic properties, the corresponding manifold admits pluriclosed and astheno-\K \ metrics (\cite[Theorem 5.6]{ciulicaotimanstanciu}, see also \cite{ciulica1});
			\item Much like for OT manifolds, the existence of complex submanifolds has been studied in \cite{ciulica2}.
		\end{itemize} 
	\end{remark}
	
	\bigskip
	
	For the main theorem in this note, the crucial fact will be that $T_M$ fibers over the circle (\cite[Proposition 2.10]{pajitnov1}), a fibration that  we now make explicit. Define
	\[
	\tau:\H\times \C^n\to \R, \qquad \tau(w,z)=\frac{\log(\Im w)}{\log \alpha}.
	\]
	Since $\tau(g_0(w,z))=\tau(w,z)+1$,	whereas the other $g_i$ ($1\leq i\leq 2n+1$) preserve $\Im w$, the map $\tau$ descends to a smooth map
	\begin{equation}
		\label{eq:fibration}
		\pi:T_M\to S^1=\R/\Z.
	\end{equation}
	This is a fibration over $S^1$ with fiber $\mathbb T^{2n+1}$ and whose monodromy is $M^T$. 
	
%	!!DE MUTAT Notice also the following local feature, which will be important in the Dolbeault computation. If $U\subset S^1$ is a sufficiently small open arc, then $\Im w$ descends to a well-defined smooth function on $\pi^{-1}(U)$. Hence, locally over $S^1$, one has
%	\begin{equation}
%		\label{eq:dbarw_exact}
%		d\overline w=-2i\,\delb(\Im w),
%	\end{equation}
%	so the anti-holomorphic form $d\overline w$ is $\delb$-exact on such open sets.
	
	\medskip
	
	Using the above property, all Betti numbers have been computed in \cite{ciulicaotimanstanciu}. To state that result, we introduce the following notation, which we will use throughout this note:
	
	\begin{definition}
		If $V$ is a vector space and $f:V\to V$ is linear, we denote by
		\[
		\bigwedge\nolimits^k f:\bigwedge\nolimits^k V\to \bigwedge\nolimits^k V
		\]
		the $k$-th exterior power of $f$, \ie
		\[
		\left( \bigwedge\nolimits^k f \right) (v_1\wedge\cdots\wedge v_k)=f(v_1)\wedge\cdots\wedge f(v_k).
		\]
	\end{definition}
	
	Then we have the following:
	
	\begin{theorem}(\cite[Theorem 4.4]{ciulicaotimanstanciu})
		\label{th:derham}
		For any $0\leq k\leq 2n+1$, one has
		\[
		h^k(T_M)=g_{k-1}+g_k,
		\]
		where $g_k$ is the geometric multiplicity of $1$ as an eigenvalue of $\bigwedge\nolimits^k M$ (with the convention that $g_{-1} = 0$).
		
		In particular:
		\begin{itemize}
			\item $h^1(T_M)=1$;
			\item $h^k(T_M)=0$, for all $1<k<2n+1$, for a generic $M$, \ie if no product of some, but not all, eigenvalues of $M$ is equal to $1$.
		\end{itemize}
	\end{theorem}
	
\subsection{Cousin groups}\label{Cousin_Subsection}
	
	We now turn to the facts about Cousin groups that will be needed later, closely following \cite[Section 2]{alextoma}.
	
	\begin{definition}
		A connected complex Lie group $X$ admitting no non-constant global holomorphic functions is called a Cousin group.
	\end{definition}
	
	\smallskip
	
	By \cite[Proposition 1.1.2]{ak}, a Cousin group of complex dimension $N$ can be written as a quotient $X=\C^N/\Lambda$, where $\Lambda$ is a discrete subgroup of $\C^N$ of rank $N+m$, with $1\leq m\leq N$. 
	
	Moreover (see \cite[Proposition 2]{vogt1}, \cite[Proposition 1]{vogt2}), $\Lambda$ may be assumed to be generated by the columns of a matrix of the form:
	\begin{equation}
		\label{eq:cousin_normal_1}
		P=\begin{pmatrix}
			O_{m,N-m} & T_{m,2m}\\
			I_{N-m} & S_{N-m,2m}
		\end{pmatrix},
	\end{equation}
	where $T_{m,2m}$ is a basis of the lattice of an $m$-dimensional complex torus and $S_{N-m,2m}$ has real entries. One may further arrange this as
	\begin{equation}
		\label{eq:cousin_normal_2}
		P=\begin{pmatrix}
			O_{m,N-m} & I_m & C+i D\\
			I_{N-m} & S_1 & S_2
		\end{pmatrix},
	\end{equation}
	where $C, D$ have real entries and $D$ is invertible. We then say that $P$ is the \textit{period matrix} of $\Lambda$.
	
	\smallskip
	
	The following criterion is essential:
	
	\begin{proposition}(\cite[Proposition 2]{vogt1})
		\label{prop:cousin_condition}
		Suppose that $X=\C^N/\Lambda$, with $\Lambda$ generated by the columns of a matrix $P$ in the normal form \eqref{eq:cousin_normal_1}. Then $X$ is a Cousin group if and only if
		\[
		{}^t\sigma S_{N-m,2m}\notin \Z^{2m}, \qquad \forall \ \sigma\in \Z^{N-m}\setminus\{0\}.
		\]
	\end{proposition}
	
	\medskip
	
	The starting point of our computation of the Dolbeault cohomology in the next section will be that the Dolbeault cohomology of some open sets in Cousin groups can be explicitly computed under certain conditions. This is the same technique that was used for computing the Dolbeault cohomology of OT manifolds in \cite{alextoma}.
	
	\begin{definition}
		\label{def:strong_disperse}
		Let $\Lambda$ be a discrete subgroup in normal form \eqref{eq:cousin_normal_2}. It is called strongly dispersive if for every $a\in (0,1)$ there exists a constant $C(a)>0$ such that
		\[
		\left\|{}^t\sigma S+{}^t\tau\right\|> C(a)a^{|\sigma|}
		\]
		for all $\sigma\in \Z^{N-m}\setminus\{0\}$ and all $\tau\in \Z^{2m}$, where $S=(S_1 \ S_2)$ is the real block from \eqref{eq:cousin_normal_2}.
	\end{definition}
	
	\begin{proposition}(\cite[Proposition 2.6]{alextoma})
		\label{prop:algebraic_strong_disperse}
		If $\Lambda$ is a discrete subgroup defining a Cousin group and such that all the entries of some period matrix are algebraic numbers, then $\Lambda$ is strongly dispersive.
	\end{proposition}
	
	\bigskip
	
	Finally, we recall the Dolbeault cohomology computation for convex domains in strongly dispersive Cousin groups, due to Otiman and Toma:
	
	\begin{theorem}(\cite[Theorem 3.1]{alextoma})
		\label{th:otiman_toma_convex}
		Let $U$ be a domain of a Cousin group $X=\C^N/\Lambda$, whose inverse image $\widetilde U$ in $\C^N$ is a convex domain. Assume that the period matrix is in the normal form \eqref{eq:cousin_normal_2} and that $\Lambda$ is strongly dispersive. Then $H^q(U,\Omega^p)$ is finite dimensional and
		\[
		\{[dz_I\wedge d\overline z_J]\ |\ I\subset \{1,\ldots,N\}, \ J\subset \{1,\ldots,m\}, \ |I|=p, \ |J|=q\}
		\]
		is a basis. In particular $\dim_\C H^q(U,\Omega^p)=\binom{N}{p}\binom{m}{q}$.
	\end{theorem}
	
	\section{Dolbeault cohomology of Endo-Pajitnov manifolds}
	\label{sec:dolbeault}
	\subsection{Preliminary steps}
	We can now turn to the main result of this note. The plan is to use the structure of $T_M$ as a fibration over $S^1$ to write it as a union of two open sets and use \ref{th:otiman_toma_convex} to compute their cohomology. The monodromy operator appearing in the resulting Mayer-Vietoris sequence can be computed in terms of the algebraic data related to the matrix $M$, leading to a short exact sequence containing the Dolbeault cohomology of the whole Endo-Pajitnov manifold. The same procedure can be reformulated in terms of computing the Leray spectral sequence of the fibration $T_M \to S^1$.
	
%	We can now turn to the main result of this note. The plan is to use the Leray spectral sequence for the fibration \eqref{eq:fibration}, using \ref{th:otiman_toma_convex} to compute the cohomology of an induced local system over $S^1$. Owing to the small dimension of the circle, this spectral sequence will degenerate quickly yielding a short exact sequence which contains the desired Dolbeault cohomology of $T_M$.
	
	With the notations from Section \ref{sec:prelim}, let
	\[
	\Lambda:=\langle u_1,\ldots,u_{2n+1}\rangle_\Z\subset \C^{n+1}
	\]
	be the translation lattice generated by $g_1,\ldots,g_{2n+1}$. Since the vectors $u_1,\ldots,u_{2n+1}$ are linearly independent over $\R$, $\Lambda$ has rank $2n+1$. Moreover, $\operatorname{span}_\R \Lambda=\R\times \C^n\subset \C\times \C^n$. Therefore the maximal complex subspace contained in $\operatorname{span}_\R\Lambda$ is
	\begin{equation}
		\label{eq:max_complex_subspace}
		L=\{0\}\times \C^n.
	\end{equation}
	
%	This is the reason why, in the local Dolbeault cohomology calculation, only $d\overline z_1,\ldots,d\overline z_n$ appear in the anti-holomorphic directions, and not $d\overline w$.
	
	\smallskip
	
	We now consider the complex Lie group
	\[
	X_\Lambda:=\C^{n+1}/\Lambda.
	\]
	
	\begin{lemma}
		\label{lem:cousin_strong}
		$X_\Lambda$ is a Cousin group and $\Lambda$ is strongly dispersive.
	\end{lemma}
	
	\begin{proof}
		Since the matrix $M$ has integer coefficients, all its eigenvalues are algebraic and all generalized eigenspaces are defined over a finite algebraic extension of $\Q$. More explicitly, if $K$ is the splitting field for the characteristic polynomial of $M$, then each generalized eigenspace $\Ker(M-\lambda I)^N$ is the solution space of a linear system with coefficients in $K$. Therefore we may choose the basis $b_1,\ldots,b_n$ of $W$ with algebraic coordinates. By \ref{rmk:choice_basis}, this does not change the biholomorphism class of $T_M$.
		
		Similarly, the real eigenvalue $\alpha$ admits a non-zero algebraic eigenvector. Since changing $a$ by a positive scalar does not change the biholomorphism class of $T_M$, again by \ref{rmk:choice_basis}, we may also suppose that $a$ has algebraic coordinates. Thus the matrix
		\begin{equation*}
			\begin{pmatrix}
				a^1 & \cdots & a^{2n+1}\\
				b_1^1 & \cdots & b_1^{2n+1}\\
				\vdots & & \vdots\\
				b_n^1 & \cdots & b_n^{2n+1}
			\end{pmatrix}
		\end{equation*}
		has algebraic entries.
		
		We now put this matrix into a normal form. After reordering the vectors $u_i$, we may suppose that $a^1\neq 0$. Make the following linear change of coordinates:
		\[
		w'=\frac{w}{a^1}, \qquad z'_k=z_k-\frac{b_k^1}{a^1}w, \quad 1\leq k\leq n.
		\]
		Then $u_1$ goes to $(1,0,\ldots,0)$ and, for $2\leq \ell\leq 2n+1$, the vector $u_\ell$ goes to
		\[
		\left(\frac{a^\ell}{a^1}, b_1^\ell-b_1^1\frac{a^\ell}{a^1}, \ldots, b_n^\ell-b_n^1\frac{a^\ell}{a^1}\right).
		\]
		All entries remain algebraic. The $z'$-components of the remaining $2n$ period vectors span $\C^n$ over $\R$; hence we may choose $n$ of them which form a complex basis of $\C^n$. Applying the inverse of the corresponding algebraic $n\times n$ matrix to the $z'$-coordinates, and reordering columns, we obtain a matrix of the form
		\begin{equation}
			\label{eq:period_matrix_EP_normal}
			P=\begin{pmatrix}
				0 & I_n & Z\\
				1 & S_1 & S_2
			\end{pmatrix},
		\end{equation}
		with the convention of placing the $w$-row as the last one. Moreover, writing $Z=C+iD$, the matrix $D$ is invertible. Thus $P$ is a period matrix as in \eqref{eq:cousin_normal_2}.
		
		We next check the Cousin condition. By \ref{prop:cousin_condition}, this is equivalent to
		\[
		k(S_1, S_2)\notin \Z^{2n}, \qquad \forall \ k\in \Z\setminus\{0\}.
		\]
		If not, then all entries of $(S_1,S_2)$ would be rational. However, by the construction of the matrix $P$, the bottom row in \eqref{eq:period_matrix_EP_normal} consists precisely of the ratios $a^\ell/a^1$. Since $a$ is an eigenvector of $\alpha$, it follows that, after rescaling, there would exist a non-zero vector $q\in \Q^{2n+1}$ such that $Mq=\alpha q$.
		Choosing an index $j$ with $q_j\neq 0$, we obtain
		\[
		\alpha=\frac{(Mq)_j}{q_j}\in \Q.
		\]
		But $\alpha$ is an algebraic integer, hence $\alpha\in \Z$. Since $M\in \textrm{SL}(2n+1,\Z)$, this forces $\alpha = \pm 1$, in contradiction with $\alpha>0$ and $\alpha\neq 1$. Therefore $X_\Lambda$ is a Cousin group.
		
		Finally, the period matrix \eqref{eq:period_matrix_EP_normal} has algebraic entries, so  \ref{prop:algebraic_strong_disperse} implies that $\Lambda$ is strongly dispersive.
	\end{proof}
	
\subsection{The main result}
	
	We are ready to state and prove the main theorem. Let
	\[
	E:=\C\langle dw,dz_1,\ldots,dz_n\rangle, \qquad F:=\C\langle d\overline z_1,\ldots,d\overline z_n\rangle.
	\]
	Define
	\[
	A:=\begin{pmatrix}
		\alpha & 0\\
		0 & R^T
	\end{pmatrix}\in \End(E), \qquad B:=\overline{R^T}\in \End(F).
	\]
	For $0\leq p\leq n+1$ and $0\leq j\leq n$, put
	\[
	V_{p,j}:=\bigwedge\nolimits^p E\otimes \bigwedge\nolimits^j F, \qquad T_{p,j}:=\bigwedge\nolimits^p A\otimes \bigwedge\nolimits^j B\in \End(V_{p,j}).
	\]
	Finally, set
	\begin{equation*}
		\label{eq:rpj_def}
		r_{p,j}:=\dim_\C \Ker(T_{p,j}-\Id),
	\end{equation*}
	with the convention that $r_{p,j}=0$ if $j<0$ or $j>n$.
	
	\begin{theorem}
		\label{th:dolbeault_main}
		The Hodge numbers of $T_M$ are given by the formula:
		\begin{equation}
			\label{eq:main_hodge_formula}
			h^{p,q}_{\delb}(T_M)=r_{p,q}+r_{p,q-1}, \ \text{for all} \ \ 0 \le p, q \le n + 1.
		\end{equation}
		More precisely, there is a natural short exact sequence
		\begin{equation}
			\label{eq:main_exact_sequence}
			0\to \coker(T_{p,q-1}-\Id)\to H^q(T_M,\Omega^p)\to \Ker(T_{p,q}-\Id)\to 0.
		\end{equation}
	\end{theorem}
	
	\begin{proof}
		Let $U\subset S^1$ be a sufficiently small open arc and let $I\subset \R$ be an interval mapping to $U$. From the definition of $\pi$ in \eqref{eq:fibration}, we have
		\[
		\pi^{-1}(U)\simeq D_I/\Lambda,
		\]
		where
		\[
		D_I=\left\{(w,z)\in \H\times \C^n \ \middle| \ \frac{\log(\Im w)}{\log \alpha}\in I\right\},
		\]
		which is clearly a convex domain in $\C^{n+1}$. 
		
		By \ref{lem:cousin_strong}, the lattice $\Lambda$ is strongly dispersive. Therefore \ref{th:otiman_toma_convex} applies for the open set $\pi^{-1}(U) \subset X_\Lambda$. Although that theorem is stated after putting the period matrix in normal form, the statement is invariant under the complex-linear changes of coordinates we used to reach that form. Returning to the original coordinates, the anti-holomorphic directions are precisely those of the maximal complex subspace \eqref{eq:max_complex_subspace}, of codimension $1$. Hence we obtain
		\begin{equation}
			\label{eq:local_cohomology}
			H^j(\pi^{-1}(U),\Omega^p)\simeq \bigwedge\nolimits^p E\otimes \bigwedge\nolimits^j F = V_{p, j},
		\end{equation}
		and the cohomology is generated by the classes
		\[
		[d\zeta_{i_1}\wedge\cdots\wedge d\zeta_{i_p}\wedge d\overline z_{j_1}\wedge\cdots\wedge d\overline z_{j_j}],
		\]
		where $\zeta_0=w$ and $\zeta_i=z_i$ ($1\leq i\leq n$), with $0\leq i_1<\cdots<i_p\leq n$ and $1\leq j_1<\cdots<j_j\leq n$.
		
		Note that the absence of $d\overline w$ from the anti-holomorphic part is precisely due to \eqref{eq:max_complex_subspace}. Equivalently, locally over $U$, the function $\Im w$ descends to $\pi^{-1}(U)$ and $d\overline w=-2i\,\delb(\Im w)$, so $d\overline w$ is $\delb$-exact locally over the base.
		
		\smallskip
		
		We now take the usual cover $S^1 = U_1 \cup U_2$ with $U_i$ arcs as above, with the intersection having two connected components. For a fixed $p$, we write the Mayer-Vietoris sequence for the Dolbeault complexes associated to $T_M = \pi^{-1}(U_1) \cup \pi^{-1}(U_2)$, with the identifications given by \eqref{eq:local_cohomology}:
		\begin{equation}
			\label{eq:MV1}
		\dots \stackrel{f_{p, q-1}}{\to} V_{p, q-1} \oplus V_{p, q-1} \to H^q(T_M, \Omega^p) \to V_{p, q} \oplus V_{p, q} \stackrel{f_{p, q}}{\to} V_{p, q} \oplus V_{p, q} \to \dots
		\end{equation}
		Since we can choose the liftings such that they agree on one of the components of the intersection, the maps $f_{p, q}$ are of type $f_{p, q}(a, b) = (a-b, a-T(b))$, where $T$ is defined by the monodromy, which we compute below.

%		Since \eqref{eq:local_cohomology} takes place for any sufficiently small $U$, and due to the  fact that restrictions preserve the above constant-form basis, the higher direct image sheaves
%		\[
%		\mathcal L_{p,j}:=R^j\pi_*\Omega^p
%		\]
%		are local systems on $S^1$, with fiber
%		\[
%		V_{p,j}=\bigwedge\nolimits^p E\otimes \bigwedge\nolimits^j F.
%		\]
		
		We use the convention that the positive generator of $\pi_1(S^1)$ acts on the fiber by pullback (as opposed to pushforward) with $g_0$. Since
		\[
		g_0^*dw=\alpha dw, \qquad g_0^*\begin{pmatrix}dz_1\\ \vdots\\ dz_n\end{pmatrix}=R^T\begin{pmatrix}dz_1\\ \vdots\\ dz_n\end{pmatrix},
		\]
		the induced action on $E$ is
		\[
		A=\begin{pmatrix}
			\alpha & 0\\
			0 & R^T
		\end{pmatrix}.
		\]
		Similarly,
		\[
		g_0^*\begin{pmatrix}d\overline z_1\\ \vdots\\ d\overline z_n\end{pmatrix}=\overline{R^T}\begin{pmatrix}d\overline z_1\\ \vdots\\ d\overline z_n\end{pmatrix},
		\]
		so the induced action on $F$ is $B=\overline{R^T}$. Therefore the monodromy is given by
		\[
		T_{p,j}=\bigwedge\nolimits^p A\otimes \bigwedge\nolimits^j B.
		\]		
		and thus $f_{p, q}(a, b) = (a-b, a - T_{p, q}(b))$ in $\eqref{eq:MV1}$. Under a change of variables $c = a - b$, this becomes $f_{p, q}(c, b) = (c, c + (I - T_{p, q}(b)))$, so \eqref{eq:main_exact_sequence} follows immediately.
		
%		We now apply the Leray spectral sequence for $\pi:T_M\to S^1$ and the sheaf $\Omega^p$:
%		\[
%		E_2^{a,b}=H^a(S^1,R^b\pi_*\Omega^p)=H^a(S^1,\mathcal L_{p,b})\Longrightarrow H^{a+b}(T_M,\Omega^p).
%		\]
%		Since the base is $S^1$, only $a=0,1$ can occur. Thus there are no non-zero differentials and the spectral sequence degenerates at $E_2$. For every $q$, we get a short exact sequence
%		\begin{equation}
%			\label{eq:leray_short_exact}
%			0\to H^1(S^1,\mathcal L_{p,q-1})\to H^q(T_M,\Omega^p)\to H^0(S^1,\mathcal L_{p,q})\to 0.
%		\end{equation}
%		Generally, if $\mathcal L_T$ is a local system on $S^1$ with fiber $V$ and monodromy $T$, then
%		\[
%		H^0(S^1,\mathcal L_T)=\Ker(T-\Id), \qquad H^1(S^1,\mathcal L_T)=\coker(T-\Id).
%		\]

		Taking dimensions, and using that
		\[
		\dim_\C\coker(T_{p,q-1}-\Id)=\dim_\C\Ker(T_{p,q-1}-\Id)
		\]
		for an endomorphism of a finite dimensional vector space, gives
		\[
		h^{p,q}_{\delb}(T_M)=r_{p,q}+r_{p,q-1},
		\]
		as claimed.
	\end{proof}
	
	\begin{remark}
		\label{rmk:inoue_numbers}
		For $n=1$, the construction gives the usual Inoue surfaces of type $S_M$ (\cite{inoue}). In this case $R=(\beta)$ and $\alpha\beta\overline\beta=1$. The only non-zero numbers $r_{p,j}$ are
		\[
		r_{0,0}=1, \qquad r_{2,1}=1.
		\]
		\ref{th:dolbeault_main} therefore gives
		\[
		h^{0,0}=h^{0,1}=h^{2,1}=h^{2,2}=1,
		\]
		and all other Hodge numbers vanish. This recovers the classical Dolbeault cohomology of Inoue surfaces.
	\end{remark}
	
	\medskip
	
	Although we never assumed $M$ to be diagonalizable in our proof, as is the case for de Rham cohomology (\cite[Corollary 4.5]{ciulicaotimanstanciu}), in the diagonalizable case the Hodge numbers are given nicely by the eigenvalues alone: 
	
	\begin{corollary}
		\label{cor:diagonalizable}
		Assume that $M$ is diagonalizable. Let
		\[
		\lambda_0=\alpha, \qquad \lambda_i=\beta_i, \quad 1\leq i\leq n,
		\]
		where the eigenvalues $\beta_i$ are listed with multiplicity. Then
		\[
		r_{p,j}=\#\left\{(I,J) \ \middle| \begin{array}{l}
			I\subset \{0,1,\ldots,n\}, \ |I|=p,\\
			J\subset \{1,\ldots,n\}, \ |J|=j,\\
			\displaystyle \prod\limits_{i\in I}\lambda_i\prod\limits_{s\in J}\overline{\beta_s}=1
		\end{array}\right\}.
		\]
		Consequently, for a generic diagonalizable $M$, \ie if no proper non-empty product of eigenvalues of $M$ is equal to $1$, the only non-zero Hodge numbers are
		\[
		h^{0,0}=h^{0,1}=h^{n+1,n}=h^{n+1,n+1}=1.
		\]
	\end{corollary}
	
	\begin{proof}
		When $M$ is diagonalizable, the matrices $A$, $B$ and therefore all $T_{p,j}$ are diagonalizable. Hence the geometric and algebraic multiplicities of the eigenvalue $1$ coincide. The eigenvalues of $A$ are $\alpha,\beta_1,\ldots,\beta_n$, while the eigenvalues of $B$ are $\overline{\beta_1},\ldots,\overline{\beta_n}$. Thus the eigenvalues of $T_{p,j}$ are precisely the products appearing in the statement. The formula for $r_{p, j}$ follows from \ref{th:dolbeault_main}.
		
		Generically, the only products equal to $1$ are the empty product and the product of all eigenvalues of $M$, since $\det M=1$. These correspond respectively to $r_{0,0}=1$ and $r_{n+1,n}=1$.
	\end{proof}
	
	\begin{corollary}
		\label{cor:hodge_decomp}
		Endo-Pajitnov manifolds satisfy Hodge decomposition at the level of dimensions \ie
		\[
		\dim_\C H^{k}_{dR}(T_M) = \sum\limits_{p+q=k} \dim_\C H_{\delb}^{p, q}(T_M).
		\]
	\end{corollary}
	
	\begin{proof}
		The vector space $\bigoplus\limits_{p+j=m}\left(\bigwedge\nolimits^p E\otimes \bigwedge\nolimits^j F\right)$ is naturally identified with $\bigwedge^m(E\oplus F)$. Under this identification, the direct sum of the operators $T_{p,j}$ is the $m$-th exterior power of the complexified action of $M$, up to transpose. Hence
		\[
		\sum\limits_{p+j=m} r_{p,j}=g_m,
		\]
		where $g_m$ is the geometric multiplicity of $1$ as an eigenvalue of $\bigwedge\nolimits^m M$.
		
		Therefore, summing \ref{th:dolbeault_main} over $p+q=k$, we get
		\[
		\sum\limits_{p+q=k}h^{p,q}_{\delb}(T_M)=\sum\limits_{p+q=k}r_{p,q}+\sum\limits_{p+q=k}r_{p,q-1}=g_k+g_{k-1}=h^k(T_M),
		\]
		which is exactly the formula from \ref{th:derham}.
	\end{proof}
		
	\bigskip
	
	\begin{remark}
		Although it is outside of the scope of the paper to do so, one can compute the Bott-Chern cohomology of $T_M$ using a similar principle.
	\end{remark}
		
	\section{Comparison with the Lie algebra cohomology}
	\label{sec:liealg}
	
	Recall that $T_M$ is a solvmanifold, $T_M \simeq \Gamma \backslash G$ for $G$ a solvable Lie group whose Lie algebra structure was computed explicitly in \cite[Proposition 5.1]{ciulicaotimanstanciu}. While the Dolbeault cohomology of $T_M$ required a bit of work and the apparatus developed in \cite{alextoma}, the Dolbeault cohomology of the Lie algebra of $G$ on the other hand is much easier to handle. We now give a result comparing the two.
	
	Let $\mathfrak g$ be the Lie algebra of the solvable Lie group $G$
	such that $T_M\simeq \Gamma\backslash G$. We use the $(1,0)$-coframe $\eta,\theta_1,\ldots,\theta_n$ of $\mathfrak g^*_\C$ given by \cite[Proposition 5.1]{ciulicaotimanstanciu}, for which
	\begin{equation}
		\label{eq:ecstrlie}
		d\eta=(\log\alpha)\,\eta\wedge\bar\eta, \qquad d\theta_k
		=
		-\sum_{j\ge k}\Delta_{kj}(\eta+\bar\eta)\wedge\theta_j, \forall 1\le k\le n,
	\end{equation}
	where $\Delta=\log R^T$.
	
	For a linear endomorphism $C\in\End(V)$, denote by
	$C^{\langle m\rangle}\in\End(\bigwedge^m V)$ the infinitesimal exterior action
	\[
	C^{\langle m\rangle}(v_1\wedge\cdots\wedge v_m)
	=
	\sum_{\ell=1}^m
	v_1\wedge\cdots\wedge C v_\ell\wedge\cdots\wedge v_m.
	\]
	Thus $\exp\bigl(C^{\langle m\rangle}\bigr) = \bigwedge\nolimits^m(\exp C)$. 
	We use the convention $C^{\langle 0\rangle}=0$.
	
	Take $\mathcal E
	:=
	\C\langle \eta,\theta_1,\ldots,\theta_n\rangle,
	\mathcal F
	:=
	\C\langle \bar\theta_1,\ldots,\bar\theta_n\rangle,
	$ and define the endomorphisms
	\[
	\mathcal A
	:=
	\begin{pmatrix}
		\log\alpha & 0\\
		0 & \Delta
	\end{pmatrix}
	\in\End(\mathcal E),
	\qquad
	\mathcal B
	:=
	\bar\Delta
	\in\End(\mathcal F).
	\]
	For $0\le p\le n+1$ and $0\le j\le n$, set
	\[
	\mathcal V_{p,j}
	:=
	\bigwedge\nolimits^p\mathcal E
	\otimes
	\bigwedge\nolimits^j\mathcal F
	\]
	and
	\[
	\mathcal D_{p,j}
	:=
	\mathcal A^{\langle p\rangle}\otimes \Id
	+
	\Id\otimes\mathcal B^{\langle j\rangle}
	\in \End(\mathcal V_{p,j}).
	\]
	We also put $\mathcal V_{p,j}=0$ and $\mathcal D_{p,j}=0$ if
	$j<0$ or $j>n$.
	
	\begin{proposition}\label{prop:lie-dolbeault-comparison}
		With the notations above:
		
		\begin{enumerate}[a)]
			\item The Lie-algebra Dolbeault cohomology of
			$\mathfrak g$ satisfies
			\[
			0
			\longrightarrow
			\coker\mathcal D_{p,q-1}
			\longrightarrow
			H_{\bar\partial}^{p,q}(\mathfrak g)
			\longrightarrow
			\ker\mathcal D_{p,q}
			\longrightarrow
			0.
			\]
			In particular, $h_{\bar\partial}^{p,q}(\mathfrak g)
			=
			s_{p,q}+s_{p,q-1}$,
			where $s_{p,j}:=\dim_\C\ker\mathcal D_{p,j}$.
			
			\item Reusing the notations of \ref{th:dolbeault_main}, we have the relation $\exp(\mathcal D_{p,j})=T_{p,j}$, so
			\[
			\ker\mathcal D_{p,j}
			\subseteq
			\ker(T_{p,j}-\Id),
			\]
			and therefore $h_{\bar\partial}^{p,q}(\mathfrak g)
			\le
			h_{\bar\partial}^{p,q}(T_M)$.
			
			\item The full Dolbeault cohomology of $T_M$ is computed by invariant
			forms if and only if
			\[
			\Spec(\mathcal D_{p,j})
			\cap
			2\pi i\bigl(\Z\setminus\{0\}\bigr)
			=
			\varnothing
			\]
			for every $0\le p\le n+1$ and every $0\le j\le n$.
	\end{enumerate}
		
	\end{proposition}
	
	\begin{proof}
		\begin{enumerate}[a)]
			\item The structure equations \eqref{eq:ecstrlie} can be rewritten in a compact form:
			\begin{equation}
				\begin{split}
					\bar\partial v = -\bar\eta\wedge\mathcal D_{p,j}v, \ &\forall v\in\mathcal V_{p,j} \\
					\bar\partial(\bar\eta\wedge u) = 0, \ &\forall u\in\mathcal V_{p,j-1}.
				\end{split}
			\end{equation}
			
			In $\bigwedge\nolimits^{p,q}\mathfrak g^*_\C =
			\mathcal V_{p,q}
			\oplus
			\bar\eta\wedge\mathcal V_{p,q-1}$, a form $v+\bar\eta\wedge u$, with
			$v\in\mathcal V_{p,q}$ and $u\in\mathcal V_{p,q-1}$, is thus
			$\bar\partial$-closed if and only if $\mathcal D_{p,q}v=0$.
			On the other hand, the $\bar\partial$-exact forms in bidegree $(p,q)$ are precisely those of type $\bar\eta\wedge\mathcal D_{p,q-1}w$ for some $w\in\mathcal V_{p,q-1}$.
			Projecting onto the first summand gives us the short exact sequence
			\[
			0
			\longrightarrow
			\coker\mathcal D_{p,q-1}
			\longrightarrow
			H_{\bar\partial}^{p,q}(\mathfrak g)
			\longrightarrow
			\ker\mathcal D_{p,q}
			\longrightarrow
			0.
			\]
			Taking dimensions, $h_{\bar\partial}^{p,q}(\mathfrak g)
			=
			\dim\ker\mathcal D_{p,q}
			+
			\dim\coker\mathcal D_{p,q-1}$. 
			Since $\mathcal D_{p,q-1}$ is an endomorphism of a finite-dimensional
			vector space, $\dim\coker\mathcal D_{p,q-1}
			=
			\dim\ker\mathcal D_{p,q-1}$. 
			Hence
			\[
			h_{\bar\partial}^{p,q}(\mathfrak g)
			=
			s_{p,q}+s_{p,q-1}.
			\]
			
			\item We now compare this with the global formula in \ref{th:dolbeault_main}. By definition,
			\[
			\exp\mathcal A
			=
			\begin{pmatrix}
				\alpha & 0\\
				0 & R^T
			\end{pmatrix}
			=
			A,
			\qquad
			\exp\mathcal B
			=
			\overline{R^T}
			=
			B.
			\]
			Using $\exp\bigl(C^{\langle m\rangle}\bigr)
			=
			\bigwedge\nolimits^m(\exp C)$,
			we obtain
			\[
			\exp(\mathcal D_{p,j})
			=
			\bigwedge\nolimits^p A
			\otimes
			\bigwedge\nolimits^j B
			=
			T_{p,j}.
			\]
			Therefore $\ker\mathcal D_{p,j}
			\subseteq
			\ker(T_{p,j}-\Id)$. The inequality $h_{\bar\partial}^{p,q}(\mathfrak g)
			\le
			h_{\bar\partial}^{p,q}(T_M)$ follows immediately.
			
			\item We now verify when equality holds. Let $D$ be any of the
			operators $\mathcal D_{p,j}$, and let $T=\exp D$. Consider the entire
			function
			\[
			Q(z)=
			\begin{cases}
				\dfrac{e^z-1}{z}, & z\ne 0,\\[1.2ex]
				1, & z=0.
			\end{cases}
			\]
			Then $T-\Id
			=
			\exp D-\Id
			=
			D\,Q(D)
			=
			Q(D)\,D$.
			The operator $Q(D)$ is invertible if and only if no eigenvalue of $D$
			lies in $2 \pi i\bigl(\Z\setminus\{0\}\bigr)$. If so, $\ker D=\ker(T-\Id)$ and $\coker D\simeq \coker(T-\Id)$.
			
			Conversely, if $D$ has an eigenvalue
			$\lambda\in 2\pi i(\Z\setminus\{0\})$, then on the generalized
			$\lambda$-eigenspace one has
			\[
			T-\Id
			=
			\exp D-\Id
			=
			\exp N-\Id,
			\]
			where $N=D-\lambda\Id$ is nilpotent. Since $\exp N-\Id=N\cdot Q(N)$ and
			$Q(N)$ is invertible, the kernel of $T-\Id$ on this generalized
			eigenspace is non-zero, whereas $D$ is invertible there. Hence in this case we have a strict inclusion $\ker D\subsetneq \ker(T-\Id)$.
			
			The final assertion follows by requiring the same condition for $D=\mathcal D_{p,j}$ 
			for all $p$ and $j$.
		\end{enumerate}		
	\end{proof}
	
	\begin{remark}
		If $M$ is diagonalizable, the preceding criterion has a particularly
		simple form. Let $\mu_i:=\log\beta_i, \ \forall 1 \le i \le n$. Then $s_{p,j}$ counts the triples
		\[
		(\varepsilon,I,J),
		\qquad
		\varepsilon\in\{0,1\},\quad
		I,J\subset\{1,\ldots,n\},
		\]
		with $\varepsilon+|I|=p, |J|=j$,
		such that
		\[
		\varepsilon\log\alpha
		+
		\sum_{i\in I}\mu_i
		+
		\sum_{\ell\in J}\bar\mu_\ell
		=
		0.
		\]
		On the other hand, $r_{p,j}$ counts the same triples for which
		\[
		\alpha^\varepsilon
		\prod_{i\in I}\beta_i
		\prod_{\ell\in J}\bar\beta_\ell
		=
		1,
		\]
		or equivalently
		\[
		\varepsilon\log\alpha
		+
		\sum_{i\in I}\mu_i
		+
		\sum_{\ell\in J}\bar\mu_\ell
		\in
		2\pi i\Z.
		\]
		Thus invariant forms compute the Dolbeault cohomology precisely when no
		non-zero integral $2\pi i$-resonance occurs among these logarithmic
		weights.
	\end{remark}
	
	\section{Rigidity under deformations}
	\label{sec:deformations}
	
	In his original paper, Inoue proved that surfaces of type $S_M$ are rigid under small deformations (\cite[Prop. 2]{inoue}). In \cite[Corollary 3.3]{angellapartonvuletescu}, the same was shown for OT manifolds of simple type.
	
	We now prove that Endo-Pajitnov manifolds are also rigid under small deformations, by applying the same techniques as in Section \ref{sec:dolbeault}:
	
	\begin{theorem}
		Any deformation of an Endo-Pajitnov manifold $T_M$ is locally trivial.
	\end{theorem}
	
	\begin{proof}
		
	By a classical result (\cite[p. 282]{kodaira-spencer}), it suffices to prove that $H^1(T_M, \Theta_{T_M}) = 0$, where $\Theta_{T_M}$ denotes the sheaf of holomorphic vector fields on $T_M$.
	
	Note first that over a small arc $U \subset S^1$, the tangent bundle $T^{1, 0} \pi^{-1}(U)$ is trivial, since the coordinate fields $\frac{\partial}{\partial w}, \frac{\partial}{\partial z_1}, ... \frac{\partial}{\partial z_n}$ descend to the quotient and form a holomorphic frame. Hence, for $V = \langle \frac{\partial}{\partial w}, \frac{\partial}{\partial z_1}, ... \frac{\partial}{\partial z_n} \rangle_\C$, we have that 
	\begin{equation}
		\label{eq:tgXU}
		H^q(\pi^{-1}(U), \Theta_{\pi^{-1}(U)}) \simeq V \otimes H^q(\pi^{-1}(U), \mathcal{O}) \simeq V \otimes \bigwedge\nolimits^q F =: W_q,
	\end{equation}
	where the last isomorphism is \eqref{eq:local_cohomology}.
	
	We again take two arcs $S^1 = U_1 \cup U_2$ and consider the corresponding Mayer-Vietoris sequence for $T_M$:
	\[
		\dots \stackrel{f_{0}}{\to} W_0 \oplus W_0 \to H^1(T_M, \Theta_{T_M}) \to W_1 \oplus W_1 \stackrel{f_1}{\to} W_1 \oplus W_1 \to \dots
	\]
	where $f_q(a, b) = (a-b, a-T_q(b))$ and $T_q$ is given by the monodromy. Exactly as in \ref{th:dolbeault_main}, this leads to the short exact sequence
	\begin{equation}
		\label{eq:MV2}
	0\to \coker(T_0-\Id)\to H^1(T_M, \Theta_{T_M})\to \Ker(T_1-\Id)\to 0.
	\end{equation}
	
	We have already determined in Section \ref{sec:dolbeault} that the action of the monodromy on $F$ is given by $B = \overline{R^T}$, thus has eigenvalues $\overline{\beta_1}, ..., \overline{\beta_n}$. On the other hand, the action on the vector fields in $V$ via pullback is contravariant, hence has eigenvalues $\alpha^{-1}, \beta_1^{-1}, ..., \beta_n^{-1}$. Since neither this last set of eigenvalues nor the product of an eigenvalue from the first set and one from the second set can equal $1$, both $T_0 - \Id$ and $T_1 - \Id$ are isomorphisms, hence it follows from \eqref{eq:MV2} that $H^1(T_M, \Theta_{T_M}) = 0$.
	\end{proof}
	
	\bigskip
	
	\textbf{Disclaimer.} The authors acknowledge the use of general-purpose large language models as an aid in some computations and for proof checking.


\bigskip
	
\begin{thebibliography}{100}
	
		\bibitem[AK01]{ak} Y. Abe, K. Kopfermann, {\it Toroidal groups}, Lecture Notes in Mathematics, vol. 1759, Springer-Verlag, Berlin (2001).
		
		\bibitem[APV20]{angellapartonvuletescu} D. Angella, M. Parton, V. Vuletescu, {\it Rigidity of Oeljeklaus-Toma manifolds}, Ann. Inst. Fourier \textbf{70(6)} (2020), 2409--2423.
        
        \bibitem[Ciu25a]{ciulica1} C. Ciulic\u a, {\it A new class of non-\K \ metrics}, Complex Manifolds \textbf{12(1)} (2025).
        
        \bibitem[Ciu25b]{ciulica2} C. Ciulic\u a, {\it Curves on Endo-Pajitnov Manifolds}, SIGMA \textbf{21} (2025), 069, 8 pages.
        
        \bibitem[COS25]{ciulicaotimanstanciu} C. Ciulic\u a, A. Otiman, M. Stanciu, {\it Special non-K\"ahler metrics on Endo-Pajitnov manifolds}, Ann. Mat. Pura Appl. \textbf{204(4)} (2025), 1425--1441.

        \bibitem[DV22]{vuliOT} \c S. Deaconu, V. Vuletescu, {\em On locally conformally \K \ metrics on Oeljeklaus-Toma manifolds}, Manuscr. Math. \textbf{171} (2023), 643-647.
		
		\bibitem[EP20]{pajitnov1} H. Endo, A. Pajitnov, {\em On generalized Inoue manifolds}, Proceedings of the International Geometry Center \textbf{13} 4 (2020), 24--39.

        \bibitem[Ino74]{inoue} M. Inoue, {\it On surfaces of Class $VII_0$}, Invent. Math. \textbf{24} (1974), 269--310.

       \bibitem[KS58]{kodaira-spencer} K. Kodaira, D. C. Spencer, \emph{A theorem of completeness for complex analytic fibre spaces}, Acta Math. \textbf{100} (1958), 281--294.

        \bibitem[OeTo05]{ot} K. Oeljeklaus, M. Toma, {\em Non-\K \ compact complex manifolds associated to number fields}, Ann. Inst. Fourier \textbf{55} 1 (2005), 161--171.
        
        \bibitem[OV]{OV:book} L. Ornea, M. Verbitsky, {\em Principles of locally conformally K\"ahler geometry}, Birkh\"auser, 2025. arXiv:2208.07188.

        \bibitem[OtTo21]{alextoma} A. Otiman, M. Toma, {\em Hodge decomposition for Cousin groups and for Oeljeklaus-Toma manifolds}, Ann. Sc. Norm. Super. Pisa Cl. Sci. \textbf{XXII} (2021), 485--503.      
        
        \bibitem[Vog82]{vogt1} C. Vogt, {\it Line bundles on toroidal groups}, J. Reine Angew. Math. \textbf{335} (1982), 197--215.
        
        \bibitem[Vog83]{vogt2} C. Vogt, {\it Two remarks concerning toroidal groups}, Manuscripta Math. \textbf{41} (1983), no. 1-3, 217--232.
        
\end{thebibliography}
\end{document}